\documentclass[11pt]{amsart}
\usepackage{verbatim}
\usepackage{rotating} 
\usepackage{amssymb}

\newtheorem{theorem}{Theorem}[section]
\newtheorem{proposition}[theorem]{Proposition}

\newtheorem{lemma}[theorem]{Lemma}

\newtheorem{cor}[theorem]{Corollary}
\newtheorem{question}[theorem]{Question}
\newtheorem{definition}[theorem]{Definition}

\newtheorem{conjecture}[theorem]{Conjecture}

\theoremstyle{plain}
\numberwithin{equation}{theorem}

\theoremstyle{remark}
\newtheorem{claim}[theorem]{Claim}
\newtheorem{remark}[theorem]{Remark}

\newcommand{\Z}{{\mathbb Z}}

\newcommand{\cA}{{\mathcal A}}

\newcommand{\fO}{\mathfrak o}
\newcommand{\fo}{\mathfrak o}
\newcommand{\fq}{\mathfrak q}
\newcommand{\fr}{\mathfrak r}

\newcommand{\fp}{\mathfrak p}
\newcommand{\tfp}{\tilde \fp}

\newcommand{\ffp}{{\mathbb F}_p}

\newcommand{\Kbar}{\overline{K}}

\DeclareMathOperator{\mymod}{mod}

\DeclareMathOperator{\Gal}{Gal}

\DeclareMathOperator{\car}{char}

\DeclareMathOperator{\bN}{\mathbb{N}}

\newcommand{\prob}{{\operatorname{Prob}}}
\newcommand{\cycle}{{\mathcal{C}}}
\newcommand{\ramification}{{\mathcal{R}}}
\newcommand{\erfc}{{\operatorname{erfc}}}
\newcommand{\A}{{\mathbb A}}

\newcommand{\bP}{{\mathbb P}}
\newcommand{\bZ}{{\mathbb Z}}

\newcommand{\bC}{{\mathbb C}}
\newcommand{\bR}{{\mathbb R}}
\newcommand{\bA}{{\mathbb A}} \newcommand{\bQ}{{\mathbb Q}}
\newcommand{\bF}{{\mathbb F}}

\newcommand{\lra}{\longrightarrow}

\newcommand{\cU}{\mathcal{U}}

\newcommand{\cT}{\mathcal{T}}
\newcommand{\cS}{\mathcal{S}}
\newcommand{\cP}{\mathcal{P}}
\newcommand{\cQ}{\mathcal{Q}}
\newcommand{\bQb}{\bar{\bQ}}
\newcommand{\bK}{\overline{K}}
\newcommand{\bL}{\overline{L}}

\usepackage{color}

\DeclareMathOperator{\lcm}{lcm}

\title{Periods of rational maps modulo primes}

\author[Benedetto, Ghioca, Hutz, Kurlberg, Scanlon, Tucker]
{Robert L.  Benedetto, Dragos Ghioca, Benjamin Hutz, P\"{a}r Kurlberg, Thomas
  Scanlon, and Thomas J. Tucker}

\date{July 8, 2011}



\address{Department of Mathematics \\
        Amherst College \\
        Amherst, MA 01002 \\
        USA}
\email{rlb@math.amherst.edu}

\address{Department of Mathematics \\
University of British Columbia\\
Vancouver, BC V6T 1Z2\\
Canada}
\email{dghioca@math.ubc.ca}

\address{ Ph.D. Program in Mathematics\\
Graduate Center of CUNY\\
365 Fifth Avenue\\
New York, NY 10016-4309
}
\email{bhutz@gc.cuny.edu}

\address{
Department of Mathematics\\
KTH\\
SE-100 44 Stockholm\\
Sweden}
\email{kurlberg@math.kth.se}

\address{Mathematics Department \\ University of California Berkeley\\
  Evans Hall \\ Berkeley CA, 94720-3840
}
\email{scanlon@math.berkeley.edu}

\address{Department of Mathematics\\
University of Rochester\\
Rochester, NY 14627\\
USA}
\email{ttucker@math.rochester.edu}

\begin{document}
\begin{abstract}
  Let $K$ be a number field, let $\varphi \in K(t)$ be a rational map
  of degree at least 2, and let $\alpha, \beta \in K$.  We show that
  if $\alpha$ is not in the forward orbit of $\beta$, then there is a
  positive proportion of primes $\fp$ of $K$ such that $\alpha \mymod
  \fp$ is not in the forward orbit of $\beta \mymod \fp$.
Moreover, we show that a similar result holds for several maps and
several points.
We also present heuristic and numerical evidence that a higher
dimensional analog of this result is unlikely to be true if we replace
$\alpha$ by a hypersurface, such as the ramification locus of a
morphism $\varphi : \bP^{n} \to \bP^{n}$.

\end{abstract}
\maketitle

\section{Introduction}\label{intro}

Let $K$ be a number field, and let $\varphi: \bP_K^1 \lra \bP_K^1$ be a
rational map of degree at least 2.
For any integer $m\geq 0$, write
$\varphi^m= \varphi\circ\varphi\circ\cdots\circ\varphi$
for the $m$-th iterate of $\varphi$ under composition.
The \emph{forward orbit} of a point $\alpha\in \bP^1(K)$
under $\varphi$
is the set $\{\varphi^m(\alpha): m\geq 0\}$.
Similarly,
the \emph{backward orbit} of $\alpha$ is the set
$\{\beta \in \bP^1(\bK) \, :\, \phi^m(\beta)
= \alpha \text{ for some } m \geq 0 \}$.
We say $\alpha$ is \emph{$\varphi$-periodic} if
$\varphi^m(\alpha)=\alpha$ for some $m\geq 1$;
the smallest such $m$ is called the (\emph{exact}) \emph{period} of $\alpha$.
More generally, we say $\alpha$ is \emph{$\varphi$-preperiodic}
if its forward orbit is finite;
if the backward orbit of $\alpha$ is finite, we say $\alpha$
is \emph{exceptional} for $\varphi$.

Given two points $\alpha, \beta \in \bP^1(K)$
such that $\beta$ is not $\varphi$-preperiodic and
$\alpha$ is not in the forward orbit of $\beta$
under $\varphi$, one might ask how many primes $\fp$ of $K$
there are such that $\alpha$ \emph{is} in the forward orbit of $\beta$
under $\varphi$ modulo $\fp$.
It follows from \cite[Lemma~4.1]{DML} that there are
infinitely many such $\fp$ unless $\alpha$ is exceptional for $\varphi$.
The same techniques
(essentially, an application of \cite[Theorem~2.2]{SilSiegel})
can be used to show that there are infinitely many
$\fp$ such that $\alpha$ is \emph{not} in the forward orbit of $\beta$
modulo $\fp$.

Odoni \cite{Odoni}, Jones \cite{Jones}, and others have shown
that the set $\cS$ of primes $\fp$
such that $\alpha$ is in the forward orbit of $\beta$
modulo $\fp$ has density zero in some cases.
However, there are cases when $\cS$ has positive density.
For example, if $K=\bQ$ and $\varphi(x) = x^3+1$,
then $\alpha=0$ is in the forward orbit of $\beta=1$
modulo any prime congruent to 2 mod 3.
More generally, one may expect such examples for
\emph{exceptional maps}; for more details, see \cite{Zieve-IMRN}.

The following is a simplified version of the main result of this paper.

\begin{theorem}\label{nicer main}
  Let $K$ be a number field,
  and let $\varphi_1, \dots, \varphi_g : \bP^1_K\lra \bP^1_K$
  be rational maps of degree at least~2.
  Let $\cA_1, \dots, \cA_g$ be finite subsets of $\bP^1(K)$ such that
  for each $i=1,\dots g$, every point in $\cA_i$ is $\varphi_i$-preperiodic.
  Let $\cT_1, \dots, \cT_g$ be finite subsets of $\bP^1(K)$ such that
  no $\cT_i$ contains any $\varphi_i$-preperiodic points.
  Then there is a positive integer $M$ and a set of primes $\cP$ of $K$
  having positive density such that
  for any $i=1,\ldots, g$,
  any $\gamma \in \cT_i$, any $\alpha \in \cA_i$,
  any $\fp\in\cP$, and any $m\geq M$,
  $$ \varphi_i^m(\gamma) \not\equiv \alpha \pmod {\fp}.$$
\end{theorem}

Theorem~\ref{nicer main} is a special case of our
main result, Theorem~\ref{proportion}, which features a weaker hypothesis:
one of the sets $\cA_i$ is allowed to contain a single non-preperiodic point.
We will prove Theorem~\ref{proportion}, and hence also
Theorem~\ref{nicer main}, in Section~\ref{main}.  The
technique is to find a prime $\fp$ at which, for each $i=1,\ldots, g$,
the expression $\varphi_i^M(x) - \alpha$ does not have a root modulo
$\fp$ for any $\alpha \in \cA_i$.  One then applies the Chebotarev
density theorem to obtain a positive density set of primes with the
desired property.

Theorem~\ref{proportion} has a number of applications to
arithmetic dynamics and to elliptic curves.  We present
two such corollaries here.  The first involves the
notion of good reduction of a rational function;
see Definition~\ref{good reduction}.

\begin{cor}
\label{positive non periodic}
Let $K$ be a number field, let $\varphi :\bP^1_K\lra \bP^1_K$ be a rational
function of degree at least~2,
and let $\alpha\in \bP^1(K)$ be a non-periodic
point of $\varphi$. Then there is a positive density set of primes
$\fp$ of $K$ at which $\varphi$ has good reduction $\varphi_{\fp}$
and such that the reduction
of $\alpha$ modulo $\fp$ is not $\varphi_{\fp}$-periodic.
\end{cor}

To state our second corollary,
we fix some notation.  If $\fp$ is a prime of a number field $K$,
$E$ is an elliptic curve defined over $K$ of good reduction
at $\fp$, and $Q\in E(K)$ is a $K$-rational point on $E$,
then $k_{\fp}=\fo_K/\fp$ is the residue field at $\fp$,
$E_{\fp}$ is the reduction of $E$ modulo $\fp$,
and $Q_{\fp}\in E_{\fp}(k_{\fp})$ is the reduction of $Q$ modulo $\fp$.

\begin{cor}
\label{positive non cyclic}
Let $K$ be a number field, let $E$ be an elliptic curve defined over $K$,
let $Q\in E(K)$ be a non-torsion point, let $q$ be a prime number,
and let $n$ be a positive integer.
Then there is a positive density set of primes $\fp$ of $K$ at
which $E$ has good reduction
and such that the order of $Q_{\fp}$ in the finite
group $E_{\fp}(k_{\fp})$ is divisible by $q^n$.
\end{cor}

Corollary~\ref{positive non cyclic} is in fact a weak version of a theorem
of Pink, who showed the following result in
\cite[Corollary 4.3]{Pink-abelian}: given an abelian
variety $A$ over a number field $K$, a point $Q\in A(K)$ such
that $\bZ\cdot Q$ is Zariski dense in $A$, and any prime power
$q^n$, there is a positive density set of primes $\fp$ of $K$ such
that the $q$-primary part of the order of $Q_{\fp}\in
A_{\fp}(k_{\fp})$ equals $q^n$.

This project originated in the summer of 2009, when four
of the authors (R.B., D.G., P.K., and T.T.)
were working to extend their results from
\cite{DML} to other cases of the Dynamical Mordell-Lang Conjecture. At
that time, T.S. suggested the general strategy for such an
extension. The details for the proposed strategy turned out to be more
complicated than originally thought, and thus later, with the help of
B.H., particularly with respect to the computations
in Section~\ref{counter}, the project was finalized.

The outline of the paper is as follows.
After some background in Section~\ref{terminology},
we state and prove Theorem~\ref{proportion}
in Section~\ref{main}.
In Section~\ref{apply}, we prove the two above corollaries,
and we present other applications
of Theorem~\ref{proportion} to some recent problems
in arithmetic dynamics.
In Section~\ref{counter} we present
evidence that a result like Theorem~\ref{proportion}
is unlikely to hold in higher dimensions.
We conclude by posing
some related questions in Section~\ref{questions}.

\section{Notation and terminology}
\label{terminology}


Let $K$ be a number field with algebraic closure $\Kbar$
and ring $\fo_K$ of algebraic integers.
Fix an isomorphism $\pi$ from $\bP^1_K$ to the
generic fibre of $\bP^1_{\fo_K}$.
By standard abuse of language,
we call a nonzero prime $\fp$ of $\fo_K$ simply a \emph{prime of $K$},
and we denote the corresponding residue field $k_{\fp}:=\fo_K/\fp$.
For each
prime $\fp$ of $K$, and for each $x\in\bP^1(K)$,
we denote by $r_{\fp}(x)$ the intersection of the Zariski closure of
$\pi(x)$ with the fibre above $\fp$ of $\bP^1_{\fo_K}$;
intuitively, $r_{\fp}(x)$ is $x$ modulo $\fp$.
The resulting map $r_{\fp}:\bP^1(K)\lra\bP^1(k_{\fp})$ is the
\emph{reduction map} at $\fp$.
Again by standard abuse of language,
we say that $\alpha\in\bP^1(K)$ is \emph{congruent} to
$\beta\in\bP^1(K)$ modulo $\fp$, and we write
$\alpha \equiv \beta \pmod \fp$,
if $r_\fp(\alpha) = r_\fp(\beta)$.

If $\varphi:\bP^1\to\bP^1$
is a morphism defined over the field $K$,
then (fixing a choice of homogeneous coordinates)
there are relatively prime homogeneous polynomials $F,G\in K[X,Y]$
of the same degree $d=\deg\varphi$ such that
$\varphi([X,Y])=[F(X,Y):G(X,Y)]$; note
that $F$ and $G$ are uniquely defined up to a nonzero constant
multiple.
(In affine coordinates, $\varphi(t)=F(t,1)/G(t,1)\in K(t)$
is a rational function in one variable.)
We can then define
the following notion of good reduction of $\varphi$,
first introduced by Morton and Silverman in~\cite{MorSil1}.

\begin{definition}
\label{good reduction}
Let $K$ be a number field,
let $\fp$ be a prime of $K$,
and let $\fo_{\fp}\subseteq K$ be the corresponding
local ring of integers.
Let
$\varphi:\bP^1\lra\bP^1$ be a morphism over $K$,
given by $\varphi([X,Y])=[F(X,Y):G(X,Y)]$, where
$F,G\in \fo_{\fp}[X,Y]$ are relatively prime homogeneous
polynomials of the same degree
such that at least one coefficient
of $F$ or $G$ is a unit in $\fO_{\fp}$.
Let $\varphi_{\fp}:=[F_{\fp},G_{\fp}]$, where
$F_{\fp},G_{\fp}\in k_{\fp}[X,Y]$ are the reductions
of $F$ and $G$ modulo $\fp$.
We say that $\varphi$ has
{\em good reduction} at $\fp$ if
$\varphi_{\fp}:\bP^1(k_{\fp})\lra\bP^1(k_{\fp})$ is a morphism of the same
degree as $\varphi$.
\end{definition}
Intuitively, the map $\varphi_{\fp}:\bP^1(k_{\fp})\to\bP^1(k_{\fp})$
in Definition~\ref{good reduction}
is the reduction of $\varphi$ modulo $\fp$.
If $\varphi\in K[t]$ is a polynomial, there is
an elementary criterion for good reduction:
$\varphi$ has good reduction at $\fp$ if and only if
all coefficients of $\varphi$ are
$\fp$-adic integers, and its leading coefficient is a $\fp$-adic unit.


We will use the following definition
in Section~\ref{apply}.
\begin{definition}
\label{def:per}
Let $K$ be a field, let $\varphi\in K(t)$ be a rational function,
and let $z\in\bP^1(\Kbar)$ be $\varphi$-periodic of period $n\geq 1$.
Then $\lambda:=(\varphi^n)'(z)$ is called the {\em multiplier} of $z$.
If $\fp$ is a prime of $K$ with associated absolute value $|\cdot |_{\fp}$,
and if $|\lambda|_{\fp} < 1$,
then $z$ is said to be {\em attracting} with respect to the prime $\fp$.
\end{definition}

In Sections~\ref{counter} and~\ref{questions} we will consider morphisms
of higher-dimensional spaces.  Therefore we note
that the multiplier $\lambda$ in Definition~\ref{def:per}
is the unique number $\lambda$ such that the induced map
$d ( \varphi^{n}):T_z {\mathbb P}^1 \lra T_z {\mathbb P}^1$
on the tangent space at $z$ is multiplication by the $1\times 1$ matrix
$[\lambda]$.

Multipliers are invariant under coordinate change.
More precisely,
if $z$ is a $\varphi$-periodic point
and $\varphi=\mu^{-1}\circ\psi\circ\mu$,
then $\mu(z)$ is a $\psi$-periodic point, and by the chain rule,
it has the same multiplier as $z$ does.
In particular, we can define the multiplier of a periodic point
at $z=\infty$ by changing coordinates.
Also by the chain rule,
the multiplier of $\varphi^{\ell}(z)$
is the same as that of $z$.

Whether or not $z$ is periodic,
we say $z$ is a {\em ramification point} or {\em critical point}
of $\varphi$ if $\varphi'(z)=0$.  If
$\varphi=\mu^{-1}\circ\psi\circ\mu$, then $z$ is a critical point
of $\varphi$ if and only if $\mu(z)$ is a critical point of $\psi$;
in particular, coordinate change can again be used to determine
whether $z=\infty$ is a critical point.

We conclude this Section by recalling the statement of the
Chebotarev Density Theorem; for more details,
see, for example, \cite{Lenstra}.

\begin{theorem}
\label{Chebotarev}
Let $L/K$ be a Galois extension of number fields, and let $G:=\Gal(L/K)$.
Let $C \subset G$ be closed under conjugation, and define
$$\Pi_C(x,L/K): =
\# \{\fp \, : \, N(\fp)\leq x, \; \fp \text{ is unramified in } L/K,
\text{ and } \sigma_{\fp} \subseteq C\},$$
where $N(\fp)$ is the $(K/\mathbb{Q})$-norm
of the prime ideal $\fp$ of $K$,
and $\sigma_{\fp}$ is the Frobenius conjugacy class
corresponding to $\fp$ in $\Gal(L/K)$. Then
$$\lim_{x\to\infty}\frac{\Pi_C(x,L/K)}{\Pi_G(x,L/K)}=
\frac{|C|}{|G|}.$$
\end{theorem}


\section{Proof of main result}\label{main}

We now state and prove our promised generalization of Theorem~\ref{nicer main}.

\begin{theorem}\label{proportion}
  Let $K$ be a number field,
  and let $\varphi_1, \dots, \varphi_g :\bP^1_K\lra \bP^1_K$
  be rational maps of degree at least~2.
  Let $\cA_1, \dots, \cA_g$ be finite subsets of $\bP^1(K)$ such that
  at most one set $\cA_i$ contains a point
  that is not $\varphi_i$-preperiodic, and such that there is
  at most one such point in that set $\cA_i$.
  Let $\cT_1, \dots, \cT_g$ be finite subsets of $\bP^1(K)$ such that
  no $\cT_i$ contains any $\varphi_i$-preperiodic points.
  Then there is a positive integer $M$ and a set of primes $\cP$ of $K$
  having positive density such that
  for any $i=1,\ldots, g$,
  any $\gamma \in \cT_i$, any $\alpha \in \cA_i$,
  any $\fp\in\cP$, and any $m\geq M$,
  $$ \varphi_i^m(\gamma) \not\equiv \alpha \pmod {\fp}.$$
\end{theorem}

We will need the following standard ramification lemma
over $\fp$-adic fields; it says, roughly, that if the field of
definition of a point in $\varphi^{-m}(\alpha)$ ramifies at $\tfp$,
then that point must be a ramification point of $\varphi^m$ modulo
$\tfp$.

\begin{lemma}
\label{ramlem}
  Let $K$ be a number field, let $\tfp$ be a prime of $\fo_{\bK}$, and let
  $\varphi: \bP^1 \lra \bP^1$ be a rational function
  defined over $K$ and of good reduction at $\fp = \tfp\cap\fo_K$
  such that $2\leq \deg\varphi < \car k_{\fp}$.
Let $\alpha\in\bP^1(K)$, let $m\geq 1$ be an integer, let
$\beta\in\varphi^{-m}(\alpha)\subseteq \bP^1(\overline{K})$,
and let $\fq:=\tfp\cap\fo_{K(\beta)}$.
If $\fq$ is ramified over $\fp$, then
$\beta$ is congruent modulo $\tfp$ to a ramification point of $\varphi^m$.
\end{lemma}

\begin{proof}
By induction, it suffices to show the lemma in the case $m=1$.
Let $|\cdot|_{\tfp}$ denote the $\tfp$-adic absolute value on $\bK$,
and let $K_{\fp}$ be the completion of $K$ with respect to
$|\cdot|_{\tfp}$.
After a change of coordinates, we may assume that $\alpha=0$
and that $|\beta|_{\tfp}\leq 1$.

Writing $\varphi=f/g$, where $f,g\in K[t]$ are relatively prime
polynomials, we have $f(\beta)=0$.
%
Since $\fq$ is ramified over $\fp$, $f$
must have at least one other
root congruent to $\beta$ modulo $\tfp$.
Thus, the reduction $f_{\fp}$ of $f$ has a multiple root at $\beta$.
However, $g_{\fp}(\beta)\neq 0$, since $\varphi$
has good reduction.
Therefore, the reduction $\varphi_{\fp}$ has
a multiple root at $\beta$, and hence $\varphi'_{\fp}(\beta)=0$.
On the other hand, because
$\deg\varphi < \car k_{\fp}$, there must be some $\gamma\in\fo_{\bK}$
such that $\varphi'_{\fp}(\gamma)\neq 0$.  It follows that
there is a root of $\varphi'$ congruent  to $\beta$ modulo $\tfp$.
\end{proof}

Next, we use the fact that our residue fields are finite to show that
if $\alpha$ is not periodic modulo a large enough prime $\fp$,
then for large $m$, there can be no roots of $\varphi^m(x) - \alpha$
modulo $\fp$.  We also obtain some extra information about our
fields of definition, which we will need in order to apply the
Chebotarev density theorem in our proof of Theorem~\ref{proportion}.

\begin{lemma}\label{lem}
  Let $K$ be a number field, let $\tfp$ be a prime of $\fo_{\bK}$, and let
  $\varphi: \bP^1 \lra \bP^1$ be a rational function
  defined over $K$ and of good reduction at $\fp = \tfp\cap\fo_K$
  such that $2\leq \deg\varphi < \car k_{\fp}$.
  Suppose that
  $\alpha\in\bP^1(K)$ is not periodic modulo $\fp$.
  Then there exists a finite extension $E$ of $K$
  with the following property:
  for any finite extension $L$ of $E$,
  there is an integer $M\in\bN$ such that for all  $m\geq M$ and all
  $\beta \in \bP^1(\bK)$ with $\varphi^m(\beta) = \alpha$,
  \begin{enumerate}
  \item $\fr$ does not ramify over $\fq$, and
\item $[\fo_{L(\beta)}/\fr : \fo_L/ \fq] > 1$,
\end{enumerate}
  where $\fr:=\tfp\cap\fo_{L(\beta)}$,
  and $\fq := \tfp\cap\fo_{L}$.
\end{lemma}


\begin{proof}
For any $\gamma \in \fo_{\bK}$,
\begin{equation} \label{one}
\text{there is at most one $j\geq 0$
such that $\varphi^j(\gamma) \equiv \alpha \pmod{\tfp}$},
\end{equation}
since $\alpha$ is not periodic modulo $\fp$.
In particular,
for each ramification point $\gamma\in\bP^1(\bK)$ of $\varphi$,
there are only finitely many integers $n\geq 0$
and points $z\in \bP^1(\bK)$ such that
$\varphi^n(z)=\alpha$ and $z \equiv \gamma \pmod {\tfp}$.
Let $E$ be the finite extension of $K$ formed by adjoining
all such points $z$.

Given any finite extension $L$ of $E$,
let $\fq = \tfp\cap\fo_L$.
Since $\bP^1(\fo_L/\fq)$ is finite, \eqref{one} implies that for
all sufficiently large $M$, the equation $\varphi^M(x) = \alpha$
has no solutions in $\bP^1(\fo_L/\fq)$.
Fix any such $M$; note that $M$ must be larger than any
of the integers $n$ in the previous paragraph.
Hence, given $m\geq M$ and $\beta\in\bP^1(\bK)$ such that
$\varphi^m(\beta) = \alpha$, we must have
$[\fo_{L(\beta)}/\fr: \fo_L / \fq] > 1$,
where $\fr = \tfp\cap \fo_{L(\beta)}$,
proving conclusion~(ii).
Furthermore, if $\beta$
is a root of $\varphi^m(x) - \alpha$, then there are two possibilities:
either (1) $\beta$ is not congruent modulo $\tfp$ to a ramification point
of $\varphi^m$,
or (2) $\varphi^j(\beta) = z$ for some $j\geq 0$
and some point $z\in\bP^1(L)$ from the previous paragraph.
In case (1), $\fr$ is unramified over $\fq$ by Lemma~\ref{ramlem}.
In case (2), choosing a minimal such $j\geq 0$,
and applying Lemma~\ref{ramlem} with $z$ in the role of $\alpha$
and $j$ in the role of $m$,
$\fr$ is again unramified over $\fq$.
Thus, in either case, conclusion~(i) holds.
\end{proof}

We now apply Lemma~\ref{lem} to a set $\cA$ of points.

\begin{proposition}\label{prop}
  Let $K$ be a number field, let $\tfp$ be a prime of $\fo_{\bK}$, and let
  $\varphi: \bP^1 \lra \bP^1$ be a rational function
  defined over $K$ and of good reduction at $\fp = \tfp\cap\fo_K$
  such that $2\leq \deg\varphi < \car k_{\fp}$.
  Let $\cA = \{ \alpha_1, \dots, \alpha_n \}$ be a
  finite subset of $\bP^1(K)$ such that
  for each $\alpha_i \in \cA$,
\begin{itemize}
\item if $\alpha_i$ is not periodic, then $\alpha_i$ is not periodic
  modulo $\fp$; and
\item if $\alpha_i$ is periodic, then $\varphi(\alpha_i) = \alpha_i$ (i.e.,
  $\alpha_i$ is fixed by $\varphi$) and the ramification index of
  $\varphi$ at $\alpha_i$ is the same modulo $\fp$ as over $K$.
\end{itemize}
Then there is a finite extension $E$ of $K$
with the following property:
for any finite extension $L$ of $E$,
there is an integer $M\in \bN$ such that for all $m\geq M$
and all $\beta \in \bP^1(\bK)$ with $\varphi^m(\beta) \in \cA$
but $\varphi^t(\beta) \notin \cA$ for all $t < m$,
\begin{enumerate}
\item $\fr$ does not ramify over $\fq$, and
\item $[\fo_{L(\beta)}/\fr : \fo_L/ \fq] > 1$,
\end{enumerate}
where $\fr:=\tfp\cap\fo_{L(\beta)}$ and $\fq:=\tfp\cap\fo_L$.
\end{proposition}

\begin{proof}
For each $\alpha_i \in \cA$ that is not periodic, we apply
Lemma~\ref{lem} and obtain a field $E_i$ with the property described
in that Lemma.
For each $\alpha_j \in \cA$ that is periodic, we apply
Lemma~\ref{lem} to each point
$\gamma_{jk} \in\varphi^{-1}(\alpha_j)\setminus\{\alpha_j\}$
and obtain a field $E_{jk}$ with the corresponding property.
To do so, of course, we must know that no $\gamma_{jk}$
is periodic modulo $\tfp$.  To see that this is true, first note
that $\gamma_{jk}\not\equiv \alpha_j \pmod{\tfp}$;
otherwise the ramification index of $\varphi$
at $\alpha_j$ would be greater modulo $\fp$ than over $K$,
contradicting our hypotheses.
Since $\alpha_j$ is fixed and
$\gamma_{jk}\not\equiv \alpha_j \pmod{\tfp}$,
it follows that $\gamma_{jk}$ is not periodic modulo
$\tfp$, as desired.


Let $E$ be the compositum of all the fields $E_i$ and $E_{jk}$.
Given any finite extension $L$ of $E$, then by our choice
of $E_i$ and $E_{jk}$, there are integers $M_i,M_{jk}\in\bN$
satisfying the conclusions of Lemma~\ref{lem}.
Set
$$M := \max_{i,j,k}(M_i, M_{jk})+1.$$
Then for any $m\geq M$ and $\beta\in\bP^1(\bK)$ such
that $\varphi^m(\beta) \in \cA$ but $\varphi^t(\beta) \notin \cA$ for
all $0\leq t < m$, we have $\varphi^{m-1}(\beta) \notin \cA$.  Hence,
$\varphi^{m-1}(\beta)$ is either some $\gamma_{jk}$ or is in
$\varphi^{-1}(\alpha_i)$ for some nonperiodic $\alpha_i$;
that is, $\beta$ is an element
either of some $\varphi^{-(m-1)}(\gamma_{jk})$
or of $\varphi^{-m}(\alpha_i)$ for some nonperiodic $\alpha_i$.
Thus, by the conclusions of Lemma~\ref{lem},
$\beta$ satisfies conditions (i) and (ii), as desired.
\end{proof}

We will now apply Proposition~\ref{prop} to several maps
$\varphi_1, \dots, \varphi_g$ at once
to obtain a proof of Theorem~\ref{proportion}.

\begin{proof}[Proof of Theorem~\ref{proportion}]
  We note first that it suffices to prove our result for a
  finite extension of $K$.
  Indeed, if $L/K$ is a finite extension
  and $\fq \cap \fo_K = \fr$ for a prime $\fq \subseteq \fo_L$,
  then $\varphi_i^m(\gamma)$ is congruent to $\alpha$ modulo $\fq$
  if and only if $\varphi_i^m(\gamma)$ is
  congruent to $\alpha$ modulo $\fr$.  Moreover, given a positive
  density set of primes $\cQ$ of $L$, the set
  $\cP = \{\fq\cap\fo_K : \fq\in\cQ\}$ also has positive density
  as a set of primes of $K$.

  We may assume, for all $i=1,\ldots, g$, that every
  $\varphi_i$-preperiodic point $\alpha\in\cA_i$ is in fact fixed
  by $\varphi_i$.  Indeed, for each such $i$ and $\alpha$,
  choose integers $j_\alpha\geq 0$ and $\ell_{\alpha}\geq 1$
  such that $\varphi_i^{j_{\alpha}}(\alpha) =
  \varphi_i^{j_{\alpha}+\ell_{\alpha}}(\alpha)$.  Set
  $j:=\max_{\alpha}\{j_{\alpha}\}$, and replace each
  $\alpha\in\cA_i$ by $\varphi_i^{j}(\alpha)$.
  Similarly, set $\ell := \lcm_{\alpha}\{\ell_{\alpha}\}$,
  and enlarge each $\cT_i = \{ \gamma_{i1}, \dots, \gamma_{is_i} \}$
  to include $\varphi_i^b(\gamma_{ic})$ for all $b=1, \dots, \ell - 1$
  and $c = 1, \dots, s_i$.  Finally, replace each $\varphi_i$
  by $\varphi_i^{\ell}$, so that for the new data, all the
  $\varphi_i$-preperiodic points in $\alpha\in\cA_i$ are fixed
  by $\varphi_i$.
  If the Theorem holds for the new data, then it holds for
  the original data, since for any $m\geq M$ and any prime $\fp$
  at which every $\varphi_i$ has good reduction,
  $\varphi_i^m(\gamma_{ij}) \equiv   \alpha \pmod {\fp}$
  implies
  $(\varphi_i^\ell)^a (\varphi_i^b(\gamma_{ij})) \equiv
  \varphi_i^j(\alpha) \pmod {\fp}$,
  writing $m+j$ as $a \ell + b$ with $a\geq 0$
  and $0\leq b < \ell$.

  We fix the following notation for the remainder of the proof.
  If there is any index $i$ such that $\cA_i$  contains a nonperiodic
  point, we may assume that this happens for $i=1$, and we denote the
  nonperiodic point by $\alpha'$.  By hypothesis, all points
  in $\cA_1\smallsetminus\{\alpha'\}$ are $\varphi_1$-preperiodic,
  and we denote them by $\alpha_{1j}$; similarly,
  for each $i\geq 2$, all points in $\cA_i$ are
  $\varphi_i$-preperiodic, and we denote them by $\alpha_{ij}$.
  By the previous paragraph, we may assume that $\varphi_i$
  fixes $\alpha_{ij}$ for all $i,j$.

Note that there are
only finitely many primes $\fp$ of bad reduction
for any $\varphi_i$, finitely many for which
$\car k_{\fp}\leq \max_i \{\deg\varphi_i\}$,
and finitely many such that the
ramification index of $\varphi_i$
at some $\alpha_{ij} \in \cA_i$ is greater modulo $\fp$
than over $K$.
On the other hand,
by \cite[Lemma~4.3]{DML}, there are infinitely many primes
$\fp$ of $K$ such that
$\alpha'$ is not $\varphi_1$-periodic modulo $\fp$.
Hence, we may choose such a prime $\fp$, and then a prime
$\tfp$ of $\fo_{\bK}$ for which $\fp=\tfp\cap\fo_K$,
that simultaneously satisfy, for each $i=1,\ldots, g$,
the hypotheses of Proposition~\ref{prop}
for $\varphi_i$ and $\cA_i$.

Applying Proposition~\ref{prop}, for each $i=1,\ldots,g$
we obtain finite extensions $E_i$ of $K$
satisfying the conclusions
of that result.  Let $L$ be the compositum
of the fields $E_1,\ldots,E_g$, and let $\fq=\tfp\cap\fo_L$.
Then for all $i=1,\ldots, g$,
all sufficiently large $M$,
and all $\beta \in \bK$ such that
$\varphi_i^M(\beta) \in \cA_i$
but $\varphi_i^t(\beta) \notin \cA_i$ for $0\leq t < M$, we have
\begin{enumerate}
\item $\fr$ does not ramify over $\fq$, and
\item $[\fo_{L(\beta)}/\fr : \fo_L/ \fq] > 1$,
\end{enumerate}
where $\fr=\tfp\cap\fo_{L(\beta)}$.
As noted at the start of this proof, it suffices to prove
the Theorem for the field $L$.

Fix such a sufficiently large integer $M$,
and let $F/L$ be the finite extension obtained
by adjoining all points $\beta\in\bP^1(\bL)$
such that for some $i = 1,\dots, g$
we have $\varphi_i^M(\beta) \in \cA_i$ but
$\varphi_i^t(\beta) \notin \cA_i$ for all $0\leq t < M$.
Note that $F/L$ is a Galois extension,
since each $\cA_i$ and each $\varphi_i$ is defined over $L$.
Moreover, by property~(i) above, $F/L$ is unramified over $\fq$.
By property~(ii), then, the Frobenius element of $\fq$
belongs to a conjugacy class of
$\Gal(F/L)$ whose members do not fix any of the points $\beta$.
By the Chebotarev density theorem (Theorem~\ref{Chebotarev}),
then, there is a positive density set of primes $\cS$ of $L$
whose Frobenius conjugacy classes
in $\Gal(F/L)$ do not fix any of the points $\beta$.

Fix any prime $\fr\in\cS$.  We make the following claim.

\begin{claim}
\label{redclaim}
Let $m\geq 0$, let $1\leq i\leq g$,
and let $z\in\bP^1(L)$ be a point such that
$\varphi_i^m(z)$ is congruent modulo $\fr$ to an element
of $\cA_i$.
Then there is some $0\leq t < M$ such that $\varphi_i^t(z)$
is congruent modulo $\fr$ to an element of $\cA_i$.
\end{claim}

To prove the claim, note first that the conclusion
is vacuous if $m<M$; thus, we may assume that $m\geq M$.
In fact, given any index $i$ and point $z$ as in the claim,
we may assume that $m$ is the minimal integer $m\geq M$
satisfying the hypothesis, namely that
$\varphi_i^m(z)=\varphi_i^M(\varphi_i^{m-M}(z))$
is congruent modulo $\fr$ to an element of $\cA_i$.
However, by the defining property of the set of primes $\cS$,
there cannot be any points $w \in \bP^1(L)$ such
that $\varphi_i^M(w)$ is congruent modulo $\fr$ to an element of
$\cA_i$ but $\varphi_i^t(w) \notin \cA_i$
for all $0\leq t < M$.
Choosing $w=\varphi_i^{m-M}(z)\in\bP^1(L)$, then,
there must be some $0\leq t < M$ such that
$\varphi_i^t(\varphi_i^{m-M}(z))$ is congruent modulo $\fr$ to
an element of $\cA_i$.  Thus,
$\varphi_i^{m-M+t}(z)$ is congruent modulo $\fr$ to
an element of $\cA_i$;
but $0\leq m-M+t <m$,
contradicting the minimality of $m$ and proving
Claim~\ref{redclaim}.

Let $\cU$ be the subset of $\cS$ consisting of primes $\fr \in \cS$
such that one or more of the following holds:
\begin{enumerate}
\item $\varphi_i^t(\gamma)\equiv \alpha_{ij} \pmod {\fr}$
for some $i=1,\ldots, g$, some $\gamma\in\cT_i$,
some $\varphi_i$-periodic $\alpha_{ij} \in \cA_i$, and some $0\leq t < M$;
or
\item $\varphi_1^t(\alpha') \equiv \alpha \pmod {\fr}$ for some
  $\alpha\in\cA_1$ and some $1\leq t \leq M$.
\end{enumerate}
Note, for each $\varphi_i$-periodic $\alpha_{ij} \in \cA_i$,
we cannot have $\varphi_i^r(\gamma) = \alpha_{ij}$
for any $r\geq 0$ and any $\gamma \in \cT_i$,
since the elements of $\cT_i$ are not $\varphi_i$-preperiodic.
(However, it \emph{is} possible that $\varphi_1^r(\gamma) = \alpha'$
for some $r$ and some $\gamma \in \cT_1$.)  Thus, $\cU$
is a finite subset of $\cS$, and hence
$\cS' := \cS \setminus \cU$ has positive density.
We will now show that the Theorem holds
for the field $L$, the integer $M$, and this set of primes $\cS'$.

Suppose there exist a prime $\fr\in\cS'$, an index
$1\leq i\leq g$, points $\alpha\in\cA_i$ and $\gamma\in\cT_i$,
and an integer $m\geq M$ such that
$\varphi_i^m(\gamma)\equiv \alpha \pmod{\fr}$.
By Claim~\ref{redclaim}, there is an integer $0\leq t < M$
and a point $\tilde{\alpha}\in\cA_i$ such that
$\varphi_i^t(\gamma)\equiv \tilde{\alpha} \pmod{\fr}$.
By property~(i) above, then, we must have $i=1$ and
$\tilde{\alpha}=\alpha'$.  Moreover, since
$\varphi_1^{m-t-1}(\varphi_1(\alpha')) \equiv \alpha\pmod{\fr}$,
and since $m-t-1\geq 0$,
Claim~\ref{redclaim} tells us that there is some
$0\leq k < M$ such that $\varphi_1^{k+1}(\alpha')$
is congruent modulo $\fr$ to an element of $\cA_1$,
contradicting property~(ii) above,
and hence proving the Theorem.
\end{proof}

\section{Applications}\label{apply}


\subsection{Proofs of the Corollaries}

We are now prepared to prove the Corollaries of Theorem~\ref{proportion}
stated in Section~\ref{intro}.

\begin{proof}[Proof of Corollary~\ref{positive non periodic}]
We begin by noting that $\varphi$ has good reduction at all but finitely
many primes $\fp$ of $K$.

If $\alpha$ is $\varphi$-preperiodic but not $\varphi$-periodic,
then for all but
finitely many primes $\fp$ of $K$, the reductions modulo $\fp$ of the
finitely many points in the forward orbit of $\alpha$ are
all distinct.  Hence $\alpha_{\fp}$ is
$\varphi_{\fp}$-preperiodic but not $\varphi_{\fp}$-periodic.
Thus, we may assume that $\alpha$ is not $\varphi$-preperiodic.

Applying
Theorem~\ref{proportion} with $g=1$, $\varphi_1=\varphi$, and
$\cT_1=\cA_1=\{\alpha\}$, there is a positive density set
of primes $\fp$ of $K$ for
which $$\varphi^m(\alpha)\not\equiv\alpha\pmod{\fp}$$ for \emph{all}
sufficiently large $m$. Hence, $\alpha_{\fp}$ is not
$\varphi_{\fp}$-periodic.
\end{proof}

\begin{proof}[Proof of Corollary~\ref{positive non cyclic}]
We begin by noting that $E$ has good reduction at
all but finitely many primes $\fp$ of $K$.
(For more details on elliptic curves, see~\cite{AEC}).
Write $E$ in Weierstrass form, and
let $x:E\to\bP^1$ be the morphism
that takes a point $P$ to the $x$-coordinate of $P$.
Let $[q]:E\to E$ denote the multiplication-by-$q$ map,
and let $\varphi\in K(x)$ be the associated Latt\`{e}s map;
that is, $\varphi$ satisfies  the identity
$x\circ [q] = \varphi\circ x$.

Since $Q$ is not torsion,
the point $[q^{n-1}]Q\in E(K)$ is also not torsion,
and therefore its $x$-coordinate
$\alpha:=x([q^{n-1}]Q)\in\bP^1(K)$
is not $\varphi$-preperiodic.  Hence,
by Corollary~\ref{positive non periodic},
there is a positive density set of primes $\fp$ of
$K$ such that the reduction $\alpha_{\fp}$ is not
$\varphi_{\fp}$-periodic.
Equivalently, $[q^{m+n-1}](Q_{\fp})\neq  [q^{n-1}]Q_{\fp}$
for all $m\geq 1$.  However, if the $q$-primary part of
the order of $Q_{\fp}$ were at most $q^{n-1}$,
then there \emph{would} be some $m\geq 1$ such that
$[q^{n-1+m}](Q_{\fp})= [q^{n-1}]Q_{\fp}$.  Thus, $q^n$ must
divide the order of $Q_{\fp}$.
\end{proof}

\subsection{Dynamical Mordell-Lang problems}
The following conjecture was proposed in \cite{Mike, newlog}.

\begin{conjecture}
[The cyclic case of the Dynamical Mordell-Lang Conjecture]
\label{DML}
Let $X$ be a quasiprojective variety defined over $\bC$,
let $\Phi$ be an endomorphism of $X$,
let $V\subset X$ be a closed subvariety, and
let $x\in X(\bC)$ be an arbitrary point.
Then the set of integers $n\in \mathbb{N}$ such that
$\Phi^n(x)\in V(\bC)$ is a union of finitely many arithmetic progressions
$\{nk+\ell\}_{n\in\mathbb{N}}$,
where $k,\ell\geq 0$ are nonnegative integers.
\end{conjecture}

Theorem~\ref{proportion} allows us to prove a few new cases of
Conjecture~\ref{DML} over number fields.

\begin{theorem}\label{ML-type}
  Let $K$ be a number field,
  let $V\subset \left(\bP^1\right)^g$ be a subvariety defined over $K$,
  let $x=(x_1,\ldots,x_g)\in(\bP^1)^g(K)$,
  and let
  $\Phi:=(\varphi_1,\dots,\varphi_g)$ act on $\left(\bP^1\right)^g$
  coordinatewise, where each $\varphi_i \in K(t)$ is a rational
  function of degree at least~2.
  Suppose that at most one $\varphi_i$ has a critical point $\alpha$
  that is not $\varphi_i$-preperiodic, and that all other critical
  points of that $\varphi_i$ are preperiodic.
  Then the set of integers $n\in \mathbb{N}$ such that
  $\Phi^n(x)\in V(\bK)$ is a union of finitely many arithmetic progressions
  $\{nk+\ell\}_{n\in\mathbb{N}}$,
  where $k,\ell\geq 0$ are nonnegative integers.
\end{theorem}

\begin{proof}
  If $x_g$ is $\varphi_g$-preperiodic, we can absorb the first finitely
  many iterates that may lie on $V$ into trivial arithmetic
  progressions $\{nk+\ell\}_{n\geq 0}$ with $k=0$.  Thus, we may
  assume that $x_g$ is $\varphi_g$-periodic, of some period $j\geq 1$.
  By restricting our attention to
  progressions $\{nk+\ell\}_{n\geq 0}$ with $j|k$, then,
  it suffices to assume $x_g$ is fixed by $\varphi_g$,
  and hence the dimension may be reduced to $g-1$.
  By induction on $g$, then, we may assume that no $x_i$ is
  $\varphi_i$-preperiodic.

  By Theorem~\ref{proportion}, there exist a
  constant $M$ and a positive proportion of primes $\fp$ of $K$ such
  that for each $i=1,\ldots, g$,
  $\varphi_i$ has good reduction at $\fp$,
  $\deg\varphi_i < \car k_{\fp}$,
  and $\varphi_i^m(x_i)$ is
  not congruent modulo $\fp$ to any critical point of $\varphi_i$
  for all $m \geq M$.
  Fix any such $\fp$, and note that
  the derivative of the reduction $(\varphi_{i,\fp})'$
  is nontrivial, because $\varphi_i$ has good reduction
  and $1\leq \deg\varphi_i<\car k_{\fp}$.
  Thus, $\varphi'_i(\varphi_i^m(x_i))$, or its appropriate
  analogue if $\varphi_i^m(x_i)$ lies in the residue class
  at $\infty$, is a $\fp$-adic unit for all $m \geq M$.
  It follows that $\varphi_i^m(x_i)\not\equiv \gamma \pmod{\fp}$
  for any attracting periodic point $\gamma$ of $\varphi_i$;
  applying \cite[Theorem 3.4]{DML} completes our proof.
\end{proof}

To state the following special case of Theorem~\ref{ML-type},
we recall that
a rational function is said to be {\it post-critically finite} if all
of its critical points are preperiodic.

\begin{cor}\label{ML-type cor}
  Let $K$ be a number field,
  let $V\subset \left(\bP^1\right)^g$ be a subvariety defined over $K$,
  let $x=(x_1,\ldots,x_g)\in(\bP^1)^g(K)$,
  and let
  $\Phi:=(\varphi_1,\dots,\varphi_g)$ act on $\left(\bP^1\right)^g$
  coordinatewise, where each $\varphi_i \in K(t)$ is
  post-critically finite and of degree at least~2.
  Then the set of integers $n\in \mathbb{N}$ such that
  $\Phi^n(x)\in V(\bK)$ is a union of finitely many arithmetic progressions
  $\{nk+\ell\}_{n\in\mathbb{N}}$,
  where $k,\ell\geq 0$ are nonnegative integers.
\end{cor}

In the case that $\varphi_i = f$
for all $i$ for some quadratic polynomial $f$, we have the following result.

\begin{theorem}\label{quadratic}
  Let $K$ be a number field,
  let $V\subset \left(\bP^1\right)^g$ be a subvariety defined over $K$,
  let $x=(x_1,\ldots,x_g)\in(\bP^1)^g(K)$,
  let $f\in K[t]$ be a quadratic polynomial,
  and let $\Phi:=(f,\dots,f)$ act on $\left(\bP^1\right)^g$
  coordinatewise.
  Then the set of integers $n\in \mathbb{N}$ such that
  $\Phi^n(x)\in V(\bK)$ is a union of finitely many arithmetic progressions
  $\{nk+\ell\}_{n\in\mathbb{N}}$,
  where $k,\ell\geq 0$ are nonnegative integers.
\end{theorem}

\begin{proof}
As in the proof of Theorem~\ref{ML-type}, we may assume that
none of $x_1,\ldots,x_g$ is preperiodic.
Let $\cT = \{x_1,\ldots,x_g\}$, and
let $\cA = \{\alpha',\infty\}$,
where $\alpha'\in K$ is the unique finite critical point of $f$.
Then the map $f$, with the finite sets $\cA$ and $\cT$, satisfies
the hypotheses of Theorem~\ref{proportion}, and the rest
of the proof is exactly like that of Theorem~\ref{ML-type}.
\end{proof}

We obtain a similar result when each $f_i$ takes the form $x^2 +
c_i$ for $c_i \in \bZ$.

\begin{theorem}\label{quadratic2}
  Let $K$ be a number field,
  let $V\subset \left(\bP^1\right)^g$ be a subvariety defined over $K$,
  let $x=(x_1,\ldots,x_g)\in (\bP^1)^g(K)$,
  let $f_i(z) = z^2 + c_i$ with $c_i \in \bZ$ for $i = 1, \dots g$,
  and let $\Phi:=(f_1,\dots,f_g)$ act on $(\bP^1)^g$
  coordinatewise.
  Then the set of integers $n\in \mathbb{N}$ such that
  $\Phi^n(x)\in V(\bQb)$ is a union of finitely many arithmetic progressions
  $\{nk+\ell\}_{n\in\mathbb{N}}$,
  where $k,\ell\geq 0$ are nonnegative integers.
\end{theorem}

\begin{proof}
  As before, we can assume that no $x_i$ is $f_i$-preperiodic.
  By \cite[Theorem 1.2]{Jones}, for each $i=1,\ldots,g$
  such that $c_i \neq 0, -1, -2$, the set
  $\cU_i$ of primes $\fp$ for which $0$ is not periodic modulo $\fp$ has
  density~1.
  (The results in \cite{Jones} are stated for $x\in\Z$,
  but the same proof works for arbitrary $x\in K$.)
  Intersecting $\cU_i$ over all such $i$ gives a set of primes $\cS_1$
  of density~1.

  For each $i=1,\ldots, g$, define
  $\cA_i=\{\infty\}$ if $c_i\neq 0, -1, -2$,
  and   $\cA_i=\{0,\infty\}$ if $c_i= 0, -1, -2$.
  Because $0$ is $f_i$-preperiodic if $c_i$ is $0$, $-1$, or $-2$,
  and because $\infty$ is exceptional and fixed for any $c_i$,
  we may apply Theorem~\ref{proportion} and conclude that
  $$\cS_2 := \{ \fp \, : \, f_i^m(x_i) \notin \cA_i
  \text{ modulo } \fp \text{ for all } m\geq M\}$$
  must have positive density, for some $M\geq 0$.
  Thus, the set $\cS = \cS_1 \cap \cS_2$
  has positive density;
  the rest of the
  proof is now like that of Theorem~\ref{ML-type}.
\end{proof}

\subsection{Newton's method at finite places}
Consider a rational function $N(x)$ of the form
$N(x) = x - \frac{f(x)}{f'(x)}$,
where $f \in K[x]$ is a polynomial of degree at least~$2$.
Given $\gamma \in K$,
let $\cS$ be the set of
primes $\fp$ of $K$ such that $\{ N^m(x) \}_{m=1}^\infty$ converges
$\fp$-adically to a root of $f$.
In \cite{Newton}, Faber and Voloch conjecture
that $\cS$ has density~0; that is,
Newton's method for approximating roots of a polynomial ``fails''
at almost all finite places of $K$.
Although we cannot use our methods to
prove this conjecture, we can prove the following result, which says that
given a finite set of nonpreperiodic points and a finite set of
rational functions $N_i(x)$ arising from Newton's method, the set of primes
at which convergence fails has positive density.
In fact, we prove that for
large enough $m$, the iterate $N_i^m(x)$ is not even in the same
residue class modulo $\fp$ as any of the roots of $f_i$.

\begin{theorem}
Let $f_1, \dots, f_g \in K[x]$ be polynomials of degree at least 2.
Let $N_i(x) = x - \frac{f_i(x)}{f_i'(x)}$ for $i= 1, \dots g$, and let
$\cT_i$ be finite subsets of $K$ such that
no $\cT_i$ contains any $N_i$-preperiodic points.  Then there is a
positive integer $M$ and a positive density set of primes $\cP$
of $K$ such that for any
$i = 1 \dots g$, any $\gamma \in \cT_i$,
any root $\alpha$ of $f_i(x)$, any $m\geq M$, and any
$\fp\in\cP$, we have
$$ N_i^m(\gamma) \not\equiv \alpha \pmod {\fp}.$$
\end{theorem}

\begin{proof}
The result is immediate from Theorem~\ref{proportion},
since each root of $f_i$ is a fixed point of $N_i$.
\end{proof}

\section{Heuristics and higher dimensions}\label{counter}

We embarked on this project hoping to prove
the cyclic case of the Dynamical Mordell-Lang Conjecture
for endomorphisms $\Phi$ of $\bP^d$ by the strategy
outlined in \cite{DML}.  (For a more general variant of this conjecture,
see \cite{JNT-Bordeaux, Mike2}.) More precisely, assuming $\Phi$ is defined
over a number field $K$, we had hoped to prove that for each
$\alpha\in \bP^d(K)$, one can \emph{always} find a prime $\fp$
of $K$ such that
for all sufficiently large $n$,
$\Phi^n(\alpha)$ is not congruent modulo $\fp$ to a
point on the ramification divisor of $\Phi$.
This is equivalent to saying that, modulo $\fp$,
the intersection of the ramification divisor and the ``periodic part''
of the forward orbit is empty. (Since any point is preperiodic modulo
$\fp$, it makes sense to divide a forward orbit into its tail and its
periodic part).  When this condition is met, even at a single prime
of suitably good reduction, one can apply the generalized
Skolem-type techniques of \cite{Jason} to prove the cyclic case of the
Dynamical Mordell-Lang Conjecture for $\Phi$ and $\alpha$.

Unfortunately,
a random map model suggests that there may be no such prime when when $d > 4$.
Roughly, the issue is that if
$\Phi: \bP^d \lra \bP^d$, then under certain assumptions of
randomness, an argument akin to the birthday paradox suggests that
the periodic part of the
forward orbit of a point under $\Phi$ should typically
be of order $p^{d/2}$.
Since the proportion of points in $\bP^d(\bF_p)$ that lie on the
ramification divisor should be about $1/p$, this means that for $d
\geq 3$
and $p$ large, the chances are very high that the periodic part of a
given forward orbit passes through the ramification divisor over
$\bF_p$.  In fact, a na\"ive argument would seem to indicate that this chance is so high when $d \geq 3$ that even
taking the product over all $p$, one is left with a nonzero chance
that the periodic part of the forward orbit of a given point passes
through the ramification divisor modulo $p$ for {\em all} $p$.
To our surprise, however, a more thorough analysis shows that the likelihood of periods
intersecting the ramification divisor modulo $p$ is dominated by {\em
  very short} cycles, namely of length $<p \log p$ (rather than
$p^{d/2}$, the  expected length of a period modulo $p$.)
This changes the dimension cutoff so that it is only when
$d \geq 5$ that there is a nonzero chance
that the periodic part of the forward orbit of a given point passes
through the ramification divisor modulo $p$ for all $p$.

In Section~\ref{probmodel}, we explain this random model in some detail
and present evidence that it is accurate in at least some cases.

\begin{remark}
The idea of using random maps to model orbit lengths is not new ---
for (generic) quadratic polynomials in one variable it is at the heart
of Pollard's rho method \cite{pollard75-rho-factoring} for factoring
integers.
Under the random map assumption, Pollard's
method factors an integer $n$ in (roughly) time $p^{1/2}$, where
$p$ is the smallest prime divisor of $n$.  As for the validity of the
random model, unfortunately not much is known.
In \cite{bach91-pollard-rho}, Bach showed that
for a randomly selected quadratic polynomial and starting point,
the random map heuristic correctly predicts the probability
of finding orbits of length about $\log p$.
Further, in \cite{silverman08-mod-p-periods},
Silverman considered general morphisms of
$\mathbb{P}^{n}$ defined over a number field, and
he showed that for any $\epsilon>0$, a random starting
point  has $\gg (\log p)^{1-\epsilon}$ distinct elements
in its orbit modulo $p$ for a full density subset of the primes;
see also \cite{Akbary-Ghioca}.
Silverman \cite{silverman08-mod-p-periods} also conjectured that the
period is greater than $p^{d/2-\epsilon}$
for a full density subset of the primes, and motivated by experimental
data, he also made a more precise conjecture
for quadratic polynomials in dimension~1.
\end{remark}

\subsection{A  probabilistic model for orbits and cycles}
\label{probmodel}

Let $X$ be a (large) finite set, and let $f : X \to X$ be a random map
in the following sense: for each $x \in X$, select the image $f(x)$ by
randomly selecting an element of $X$, with uniform distribution.
\subsubsection{Cycle lengths}
Fix a starting point $x_0 \in X$,
and inductively define $x_{n+1} = f(x_n)$.
Since $X$ is finite, $x_0$ is necessarily preperiodic.
Let $\tau$ be the collision time for the orbit $(x_0,x_1,\ldots)$,
i.e., $\tau$ is the smallest positive integer
such that $x_\tau = x_s$ for some
$s<\tau$.
For any integer $k\geq 0$,
we have
$x_j \not \in \{x_0,x_1,\ldots, x_{j-1}\}$ for all $j \leq k$
if and only if $\tau > k$.
Thus, the randomness assumption on $f$ implies that
$$
\prob(\tau > k) = \prod_{j=1}^k  \Big(1-\frac{j}{|X|}\Big)
=
\exp\Big[ \sum_{j=1}^k \log \Big(1-\frac{j}{|X|}\Big) \Big],
$$
as in the birthday paradox.
From the Taylor series expansion $\log(1-x)=-(x+x^2/2+x^3/3+\ldots)$,
we deduce the inequality
\begin{equation}
  \label{eq:coll-prob-bound}
\prob(\tau > k) \leq \exp \left( - \frac{k(k+1)}{2|X|} \right),
\end{equation}
and similarly we find that for $k = o(|X|^{2/3})$,
\begin{equation}
  \label{eq:collision-asymptotics}
\prob(\tau > k) = \exp \left( - \frac{k^{2}}{2|X|} \right) \cdot (1+o(1)),
\end{equation}
since
\begin{multline}
\prob(\tau > k) = \exp \left( - \frac{k(k+1)}{2|X|} + O(k^{3}/|X|^{2})
\right)
\\=
\exp \left( -\frac{k^{2}}{2|X|} + O(k/|X|+k^{3}/|X|^{2}) \right).
\end{multline}

In addition, if we let $\alpha(k) := \prob(\tau > k-1)$, then
\begin{align*}
\prob(\tau = k)
&= \prob(\tau > k-1) - \prob(\tau > k)
\\
&=
\prod_{j=1}^{k-1}  \Big(1-\frac{j}{|X|}\Big)
- \prod_{j=1}^k  \Big(1-\frac{j}{|X|}\Big)
\\
&=
\Big[1 - \Big(1-\frac{k}{|X|}\Big)\Big]
\cdot \prod_{j=1}^{k-1}  \Big(1-\frac{j}{|X|}\Big)  =
\frac{k}{|X|} \cdot \alpha(k).
\end{align*}

Define $\cycle:=\{x_s=x_{\tau}, x_{s+1}, \ldots , x_{\tau-1}\}$
to be the periodic part of the orbit of $x_0$.
Conditioning on $\tau=k$,
the random map assumption implies that
$x_{k}$ is uniformly selected among $\{x_{0}, \ldots, x_{k-1}\}$,
and hence
$$
\prob\Big(|\cycle|=\ell \Big| \tau = k\Big) = \frac{1}{k}
\quad
\text{for any}
\quad
\ell\leq k.
$$
The cycle length probability may thus be written as
\begin{multline}
\label{eq:cycle-prob-with-alpha}
\prob(|\cycle|=\ell) =
\sum_{k \geq \ell} \prob(|\cycle|=\ell \big| \tau = k) \cdot
\prob(\tau = k)
\\=
\sum_{k \geq \ell} \frac{1}{k} \cdot \prob(\tau = k)
=
\sum_{k \geq \ell} \frac{1}{k} \cdot \frac{k}{|X|} \cdot  \alpha(k) =
\frac{1}{|X|} \sum_{k \geq \ell} \alpha(k).
\end{multline}

Before stating the next Lemma we recall that the {\it Gaussian error
  function} $\erfc$ is defined
as (cf. \cite[Chapter 7]{Abramowitz-Stegun})
$$ \erfc(s):=\frac{2}{\sqrt{\pi}} \int_{s}^{\infty} e^{-t^{2}} \, dt.$$
\begin{lemma}
\label{lem:cycle-length}
If  $\ell = o(|X|^{2/3})$ then, as $|X| \to \infty$,
$$
  \label{eq:cycle-length-probability}
\prob(|\cycle|=\ell) =
\sqrt{\frac{\pi}{2|X|}} \cdot
\left( \erfc
\left( \ell/\sqrt{2|X|}
\right) + o(1)\right).
$$
\end{lemma}
\begin{proof} By \eqref{eq:cycle-prob-with-alpha}, we find that
\begin{equation}
\label{eq:cycle-prob-as-sum}
\prob(|\cycle|=\ell) =
\frac{1}{|X|} \sum_{k = \ell}^{|X|} \alpha(k)
=
\frac{1}{|X|}
\left(
\sum_{k = 1}^{|X|} \alpha(k)
-
\sum_{1 \leq k < \ell} \alpha(k) \right).
\end{equation}
We begin by evaluating the first sum.  Recalling that
$\alpha(k)=\prob(\tau>k-1)$,
if $k = o(|X|^{2/3})$, then by
\eqref{eq:collision-asymptotics}, we have
$$
\alpha(k) = \exp\big( -k^{2}/(2|X|)\big) \cdot \big(1+o(1)\big).
$$
Moreover, by \eqref{eq:coll-prob-bound} the inequality
$$
\alpha(k)  \ll \exp\big( -k^{2}/(3|X|)\big)
$$
holds for  $k \geq 1$.
Thus, setting $Q(T) := T^{2/3}/\log T$, we have
\begin{multline}
\label{eq:long-alpha-sum}
\sum_{k = 1}^{|X|} \alpha(k)
=
\sum_{1 \leq k \leq Q(|X|)} \alpha(k)
+
\sum_{Q(|X|) < k \leq |X|} \alpha(k)
\\=
\big(1+o(1)\big) \cdot \sum_{1 \leq k \leq Q(|X|)} e^{-k^{2}/(2|X|)}
+
O \left(
\int_{Q(|X|)-1}^{\infty}  e^{ -t^{2}/(3|X|)}   \, dt
\right).
\end{multline}
To show that the contribution from the integral is negligible, we note
the inequality (valid for all $A,B>0$)
$$
\int_A^{\infty} e^{-t^{2}/B} \, dt  =
\sqrt{B} \int_{A/\sqrt{B}}^{\infty} e^{-s^{2}} \, ds \leq
\frac{B}{A} \int_{A/\sqrt{B}}^{\infty} s  e^{-s^{2}}
\, ds
=
\frac{B}{2A} e^{-A^2/B}.
$$
Thus,
\begin{multline}
\label{eq:tail-integral-bound}
\int_{Q(|X|)-1}^{\infty}  e^{ -t^{2}/(3|X|)}   \, dt
\leq
\frac{3|X|}{2Q(|X|)-2} \exp
\left(
- \frac{\big(Q(|X|)-1\big)^{2}}{3|X|}
\right)
\\
\ll
|X|^{1/3}\log|X| \exp\left(
-\frac{|X|^{1/3}}{3(\log |X|)^2}\right)
=o(1)
\end{multline}
as $|X|\to \infty$.

Meanwhile, note that for any $L\geq 1$,
\begin{multline*}
\sum_{1 \leq k \leq L} e^{-k^{2}/(2|X|)}
=
\sqrt{2|X|} \cdot \left[
\int_{0}^{\lfloor L\rfloor/\sqrt{2|X|}} e^{-t^{2}} \, dt
+ O \left( \frac{\lfloor L \rfloor}{2|X|} \right) \right]
\\=\sqrt{2|X|} \cdot
\int_{0}^{L/\sqrt{2|X|}} e^{-t^{2}} \, dt
+ O \left( \frac{ L}{\sqrt{2|X|}} + 1 \right).
\end{multline*}
by interpreting the sum
as a $1/\sqrt{2|X|}$-spaced Riemann sum
approximation of an integral and noting that
$|e^{-s^2} - e^{-t^2}|\leq |s-t|$ for all $s,t\in\bR$.
Thus, the sum in the right side of \eqref{eq:long-alpha-sum} is
\begin{multline*}
\sum_{1 \leq k \leq Q(|X|)} e^{-k^{2}/(2|X|)}
=
\sqrt{2|X|} \cdot
\int_{0}^{Q(|X|)/\sqrt{2|X|}} e^{-t^{2}} \, dt
+ O \left( \frac{Q(|X|)}{\sqrt{2|X|}}  + 1\right)
\\=
\sqrt{2|X|} \cdot
\left( \int_{0}^{\infty} e^{-t^{2}} \, dt + o(1) \right),
\end{multline*}
and the second sum on the right side of
\eqref{eq:cycle-prob-as-sum} is
\begin{multline*}
\sum_{1 \leq k < \ell} \alpha(k)  =
\big(1+o(1)\big) \cdot \sum_{k=1}^{\ell-1} e^{-k^{2}/(2|X|)}
\\=
\big(1+o(1)\big) \cdot
\left(
\sqrt{2 |X|} \int_{0}^{\ell/\sqrt{2 |X|}} e^{-t^{2}} \, dt +
O\Big(\frac{\ell}{\sqrt{2|X|}}+1)
\right).
\end{multline*}
Combining equations
\eqref{eq:cycle-prob-as-sum},
\eqref{eq:long-alpha-sum}, and
\eqref{eq:tail-integral-bound} with the above
Riemann sum estimates, and recalling that
$\erfc(s)=\frac{2}{\sqrt{\pi}} \int_{s}^{\infty} e^{-t^{2}} \, dt$,
we have
\begin{multline*}
\prob(|\cycle| = \ell) = \frac{\sqrt{2}
\cdot \int_{\ell/\sqrt{2 |X|}}^{\infty} e^{-t^{2}} \, dt +
o(1)}{|X|^{1/2}}
\\=
\sqrt{\frac{\pi}{2|X|}} \cdot
\left( \erfc \left( \ell/\sqrt{2|X|} \right) + o(1)\right).
\qedhere
\end{multline*}
\end{proof}


\subsubsection{Cycles intersecting the ramification locus}
\label{sec:prob-cycl-inters}

We now specialize to polynomial maps: let $\phi : \A^{d}(\Z) \to
\A^{d}(\Z)$ be a polynomial map such that its Jacobian matrix $d\phi$
has non-constant determinant, and fix a starting point
$x_{0} \in \A^{d}(\Z)$.
Given a prime $p$, let $X_{p} = \A^{d}(\ffp)$,
and denote by $\phi_{p} : X_{p} \to X_{p}$ the reduction
of $\phi$ modulo $p$.  Further, let $\cycle_{p}$ denote the periodic
part of the forward orbit of $x_{0}$ under $\phi_{p}$, and let
$\ramification_{p}$ denote the ramification locus of $\phi_{p}$.
We say $\phi$ has {\em random map behavior} modulo $p$ if the
following two conditions hold:
\begin{itemize}
\item $|\cycle_{p}|$ has the same probability distribution as
  the cycle length of a random map on a set of size $p^{d}$.

\item The probability of a collection of distinct points
  $y_{1},\ldots,y_{k} \in \cycle_{p}$ all belonging to
  $\ramification_{p}$ is $1/p^{k}$.

\end{itemize}

By the Weil bounds, $|\ramification_{p}| = p^{d-1}\cdot (1+o(1))$,
since $\ramification_{p}$ is a hypersurface (assumed irreducible for
simplicity) defined by the vanishing of the determinant of
the Jacobian of the map $\phi_{p}$.  Thus, the main thrust of the second
assumption above is that the sets $\cycle_{p}$ and $\ramification_{p}$ are
suitably independent.

\begin{proposition}
\label{prop:empty-intersection-probability}
  Assume that the polynomial map $\phi : \A^d(\Z)\to\A^d(\Z)$
  has random map behavior modulo every sufficiently large prime $p$.
  If $d \geq 3$, then
$$
\prob( \cycle_{p} \cap \ramification_{p} = \emptyset)
=
\frac{\sqrt{\pi/2}}{p^{d/2-1}} \cdot (1+o(1))
\quad
\text{as }
p\to\infty.
$$
\end{proposition}
\begin{proof}
  Fix a large enough prime $p$.
  For simplicity of notation, we will write $\cycle$ and
  $\ramification$ instead of $\cycle_{p}$ and $\ramification_{p}$.
  Conditioning on
  the cycle length $|\cycle|$ being equal to $\ell$, we
  find that $\prob\Big( \cycle \cap \ramification = \emptyset \Big|
  |\cycle|=\ell\Big) = (1-1/p)^{\ell}$, and hence
$$
\prob( \cycle \cap \ramification = \emptyset)
=
\sum_{\ell=1}^{p^{d}} (1-1/p)^{\ell} \cdot \prob(|\cycle|=\ell).
$$
We start by bounding the contribution from the large cycles.
Since $(1-1/p)^{\ell}$ is
a decreasing function of $\ell$
and
$\sum_{\ell=1}^{p^{d}} \prob(|\cycle|=\ell) = 1$, we have
$$
\sum_{\ell \geq d p \log p}^{p^{d}} \Big(1-\frac{1}{p}\Big)^{\ell} \cdot
\prob(|\cycle|=\ell)
\leq
\Big(1-\frac{1}{p}\Big)^{d p \log p}
\\
\ll \exp( - d \log p ) = p^{-d}.
$$

To determine the contribution from the short cycles we argue as
follows.  By Lemma~\ref{lem:cycle-length},
for $\ell \leq d p \log p = o(p^{d/2})$, we have
$$
\prob(|\cycle|=\ell)
=
\sqrt{\frac{\pi}{2p^{d}}} \cdot
\left( \erfc
\left( \ell/\sqrt{2p^{d}}
\right) + o(1)\right)
=
\sqrt{\frac{\pi}{2p^{d}}} \cdot (1+o(1))
$$
since $\erfc(0)=1$.
Hence,
$$
\sum_{\ell=1}^{dp\log p}
\Big(1-\frac{1}{p}\Big)^{\ell} \cdot \prob(|\cycle|=\ell)
=
\big(1+o(1)\big) \cdot \sqrt{\frac{\pi}{2p^{d}}} \cdot
\sum_{\ell=1}^{dp\log p}
\sum_{1 \leq \ell < d p \log p}
\Big(1-\frac{1}{p}\Big)^{\ell},
$$
which, on summing the geometric series, equals
\[
\big(1+o(1)\big) \cdot \sqrt{\frac{\pi}{2p^{d}}} \cdot
\frac{1-O(p^{-d})}{1-(1-1/p)}
=
(\sqrt{\pi/2}+o(1))  \cdot p^{1-d/2}.
\qedhere
\]
\end{proof}

\begin{remark}
\label{rem:d=2prob}
For $d=2$ a similar argument gives that as $p \to \infty$,
\begin{multline*}
\prob(\cycle_{p} \cap \ramification_{p} = \emptyset)
= \big(1+o(1)\big) \cdot
\sum_{\ell \leq p^{2} } (1-1/p)^{\ell} \sqrt{\pi/(2p^{2})}
\erfc(\ell/\sqrt{2p^{2}})
\\= \big(1+o(1)\big) \cdot
\sqrt{\pi} \int_{0}^{\infty} e^{-\sqrt{2}t} \erfc(t) \, dt
\approx (1+o(1)) \cdot 0.598 .
\end{multline*}
For $d=1$ it is easy to see that $\prob(\cycle_{p} \cap
\ramification_{p} = \emptyset) = 1+o(1)$ as $p \to \infty$ as follows:
since
$|\ramification_{p}| = O(1)$,
$$
\prob( \cycle_{p} \cap \ramification_{p} = \emptyset \Big|
|\cycle_{p}| < p^{1/2} \log p) > (1-O(1/p))^{p^{1/2} \log p} = 1+o(1)
$$
and, by \eqref{eq:coll-prob-bound},
\begin{multline*}
\prob( |\cycle_{p}| \geq p^{1/2}
\log p) \leq \prob( \tau \geq p^{1/2} \log p)
\\=
\exp( -(\log p)^{2}/2)(1+o(1)) = o(1).
\end{multline*}

\end{remark}

\subsubsection{Global probabilities in higher dimensions}
\label{sec:glob-prob-high}

Since it is enough to find {\em one} prime $p$ for which $\cycle_{p}
\cap \ramification_{p} = \emptyset$, the ``probability'' that our
approach fails --- assuming good reduction of the map for all primes, as
well as the random map model being applicable and that the ``events''
$\cycle_{p} \cap \ramification_{p} = \emptyset$ are independent for
different $p$ --- in dimension $d \geq 3$ is, by
Proposition~\ref{prop:empty-intersection-probability}, given by an
Euler product of the form
$$
\prod_{p}
\left(
1-\frac{\sqrt{\pi/2}+o(1)}{p^{d/2-1}}
\right)
=
\prod_{p}
\left(
1-O(p^{1-d/2})
\right).
$$
Since the product diverges to zero if $d = 3,4$, we would in this case expect
to find at least one (if not infinitely many) primes for which
$\cycle_{p} \cap \ramification_{p} = \emptyset$.
On the other hand, if $d \geq 5$ the product converges, and hence
there is a non-vanishing probability that $\cycle_{p} \cap
\ramification_{p} \neq \emptyset$ {\em for all} primes.

\subsection{Numerical evidence for the random model}
\label{sec:numer-exper}


Letting $\tilde{c} := \frac{|C|}{\sqrt{2|X|}}$ denote the
normalized cycle length, it is straightforward to deduce from
Lemma~\ref{lem:cycle-length} that the probability density function of
$\tilde{c}$ is given by
\begin{equation}
  \label{eq:normalized-cycle-pdf}
g(s) = \sqrt{\pi} \cdot \erfc(s),
\end{equation}
i.e., that
$$
\prob( \tilde{c} \leq t ) = \int_{0}^{t} g(s) \, ds.
$$
In this section, we shall compare observed cycle lengths
with this prediction in dimensions one and three.

\subsubsection{Cycle lengths in dimension $d=1$}
Consider the map $x \to f(x)$, where $f(x) = x^{2}+x+2$,
with starting point $x_0=1$.
For each prime $p < 100000$, we computed
the normalized cycle length $\tilde{c}_{p} := |\cycle_{p}|/\sqrt{2p}$.
A histogram plot of $\{\tilde{c}_{p}\}_{p<N}$ for the resulting data
appears in Figure~\ref{fig:histdim1}, along with the probability
density function $g(t) = \sqrt{\pi} \cdot \erfc(t)$
from \eqref{eq:normalized-cycle-pdf}.


\begin{figure}[h!]
\centering
\includegraphics[width=6cm,]{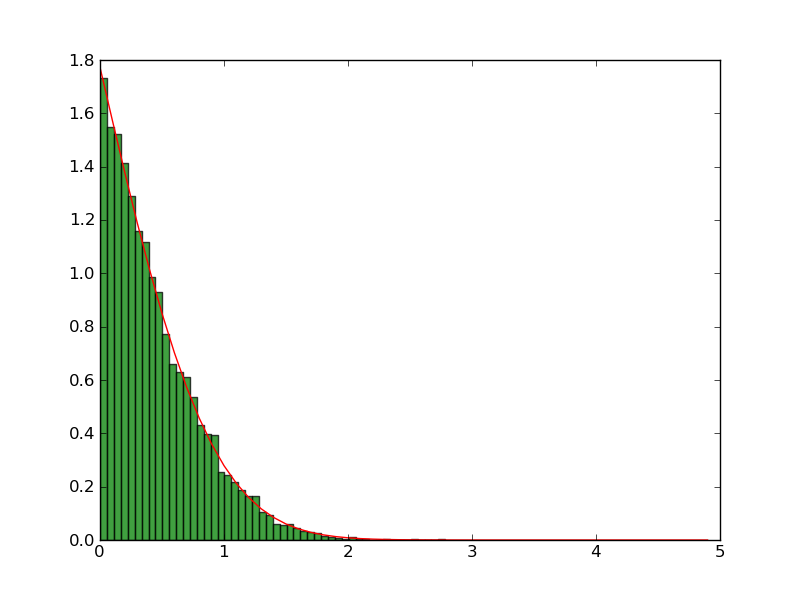}
\caption{\label{fig:histdim1}Normalized cycle length statistics for
  $p<100000$.}
\end{figure}

\subsubsection{Cycle lengths in dimension $d=3$}
Next,
consider the map
$$
(x_{1},x_{2},x_{3}) \to (f_{1}(x_{1},x_{2},x_{3}),
f_{2}(x_{1},x_{2},x_{3}), f_{3}(x_{1},x_{2},x_{3})) ,
$$
where
$$
f_{1}(x_{1},x_{2},x_{3}) =
x_1^{2} + 2 x_1 x_2 -3 x_1 x_3 + 4x_2^{2} + 5x_2x_3 + 6x_3^{2} +
7x_1 + 8x_2 + 9x_3 + 11,
$$
$$
f_{2}(x_{1},x_{2},x_{3}) =
2x_1^{2} + 3 x_1 x_2 +4 x_1 x_3 + 5x_2^{2} + 6x_2x_3 + 10x_3^{2} +
1x_1 + 8x_2 + 3x_3 + 7,
$$
$$
f_{3}(x_{1},x_{2},x_{3}) =
3x_1^{2} + 4 x_1 x_2 +5 x_1 x_3 + 6x_2^{2} + 17x_2x_3 + 11x_3^{2} +
2x_1 + 8x_2 + 5x_3 + 121.
$$

%

With starting point $x_0=(1,2,3)$,
we proceed as we did in dimension one, except that
now the normalized cycle length is
$\tilde{c}_{p} := |\cycle_{p}|/\sqrt{2p^3}$,
and we consider only $p<21000$.  The resulting histogram
and expected probability density function appear in
Figure~\ref{fig:histdim3}.

\begin{figure}[h!]
\centering
\includegraphics[width=6cm,]{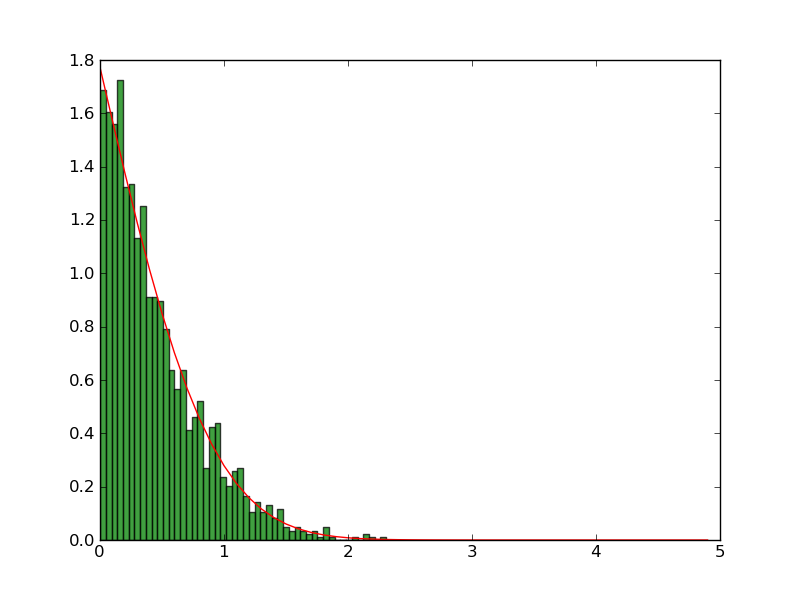}
\caption{\label{fig:histdim3}Normalized cycle length statistics for $p<21000$.}
\end{figure}

\subsubsection{Probability that $\ramification_{p} \cap \cycle_{p} = \emptyset$}

To further check the accuracy of the predictions of
Section~\ref{sec:prob-cycl-inters},
we randomly generated $50000$  degree~$2$ polynomial maps
$\varphi:\bA^d \to \bA^d$ with integer coefficients
and checked how often the forward orbit of $(0,\ldots,0)$
had a periodic point on the ramification divisor modulo~$p$;
see Figure~\ref{figure3}.
In both cases, the data shows that as $p$ increases, the probability
quickly goes to $0$ in dimension~1, remains roughly constant in
dimension~2 (compare with Remark~\ref{rem:d=2prob} and note that $1-0.598
\simeq 0.4$), and quickly goes to $1$ in dimension~$3$, as suggested by
our model.

\begin{figure}[h!]
    \includegraphics[angle=0,scale=.38]{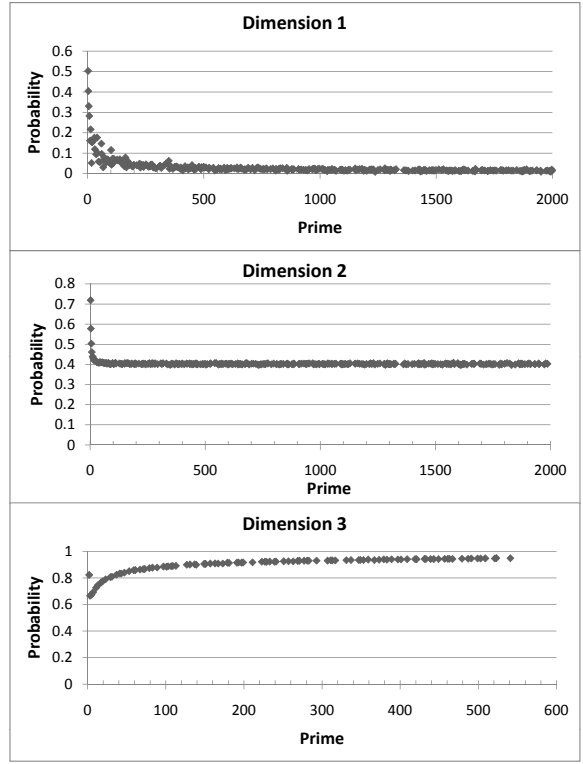}
    \caption{Probability that the forward orbit of a point
      has a periodic point lying on the ramification divisor modulo $p$}
    \label{figure3}
\end{figure}

%

For $d=3$, we also compared the number of primes $p<N$ for which $\cycle_{p}
\cap \ramification_{p} = \emptyset$ with the prediction given by the
random map model.  Thus, given $p$, let $X_{p} = 1$ if $\cycle_{p}
\cap \ramification_{p} = \emptyset$, or $X_p=0$ otherwise.
According to Proposition~\ref{prop:empty-intersection-probability},
$\prob(X_{p} = 1 )$ should be
$\sqrt{\pi/(2p)} \cdot (1+o(1))$.  To test this
prediction,
Figure~\ref{fig:random_model} compares
$$
S(N) := \sum_{p < N} X_{p}
$$
with its expected value $\sum_{p < N} \sqrt{\pi/(2p)}$. For the first
$5000$ primes $p$ we estimated $\prob(X_{p}=1)$ by taking a collection
of $50$ different polynomial maps of degree two, and for each $p$ we
computed the proportion of maps for which $\cycle_{p} \cap
\ramification_{p} = \emptyset$.
Once again, as Figure~\ref{fig:random_model} shows, the the data
agrees very closely with the predictions of the random map model.

\begin{figure}[h!]
    \includegraphics[angle=0,scale=.38]{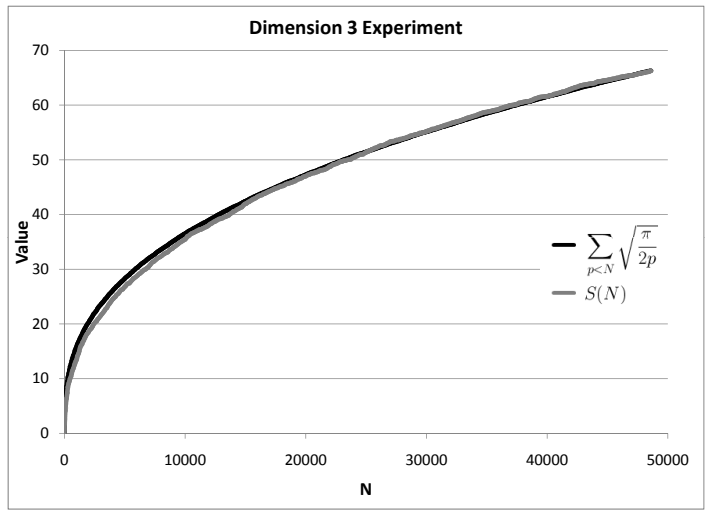}
    \caption{$S(N)$ compared to $\sum_{p < N} \sqrt{\pi/(2p)}$.}
    \label{fig:random_model}
\end{figure}


\section{Further questions}
\label{questions}

It is natural to ask whether Theorem~\ref{proportion}
is true if we remove the
restriction that at most one element of $\bigcup_{i=1}^g \cA_i$ is not
$\varphi_i$-preperiodic.  Unfortunately, our method does not
extend even to the case that $g =1 $ and $\cA_1 =\{ \alpha_1, \alpha_2 \}$
if neither $\alpha_1$ nor $\alpha_2$ is $\varphi_1$-preperiodic.
Indeed, the proof of Theorem~\ref{proportion} uses
\cite[Lemma~4.3]{DML}, which relies on Roth's theorem and
\cite{SilSiegel}, and it is not at all clear how to extend those methods
to the case of more than one wandering point.

Here is one particularly simple question that we have been
unable to treat with our methods.

\begin{question}\label{two_pts}
  Let $K$ be a number field, let $\varphi: \bP_K^1 \lra \bP_K^1$
  be a rational map of degree at least 2, and let $\gamma_1, \gamma_2
  \in \bP^1(K)$ be nonperiodic for $\varphi$. Are there infinitely
  primes $\fp$ of $K$ such that neither $\gamma_1$ nor $\gamma_2$ is
  periodic modulo $\fp$?  Is the density of such primes positive?
\end{question}

The answer to Question~\ref{two_pts} is trivially ``yes'' if
both $\gamma_1$ and $\gamma_2$ are preperiodic.  In addition,
Theorem~\ref{proportion} also gives a positive answer if
only one of the points is preperiodic, or if $\phi^m(\gamma_1)=\gamma_2$
for some $m\geq 0$.  Other than these special arrangements, however,
we know of very few cases in which we can answer
Question~\ref{two_pts}.
One such case, again with a positive answer, is the case that
$\gamma_1=0$ and $\varphi(x)=x^2 + c$, where $c$ an integer other than
$0$ or $-1$.  The proof is like that of Theorem~\ref{quadratic2}:
since the intersection of
a set of positive density with a set of density~1 has positive
density, the result follows by combining Theorem~\ref{proportion}
with the results of \cite{Jones}.

In a different direction,
one might also ask for a higher-dimensional version of
Theorem~\ref{proportion} involving points rather than hypersurfaces.

\begin{question}\label{q1}
Let $K$ be a number field, let $N  > 1$, let $\Phi: \bP_K^N \lra \bP_K^N$
  be a morphism of degree at least 2,
  and let $\gamma_1, \gamma_2 \in \bP^N(K)$.
  Suppose that there is no $m\geq 0$ such that
  $\Phi^m(\gamma_1) = \gamma_2$.  Are there infinitely many primes
  $\fp$ of $K$ such that
  $\Phi^m(\gamma_1) \not\equiv \gamma_2 \pmod{\fp}$ for all $m$?
  Is the density of such primes positive?
\end{question}

In contrast to the case of orbits intersecting hypersurfaces
as in Section~\ref{counter},
it appears to be \emph{less} likely that
the orbit of a point passes through another point modulo a prime in
higher dimensions than in dimension 1.
Indeed, the same orbit length heuristics suggest
that at a given prime $p$ of $\bQ$, there is a
$\frac{1}{p^{N/2}}$ chance that the orbit of $\gamma_1$ meets
$\gamma_2$ modulo $p$.
Because $\prod_p(1 - \frac{1}{p^{N/2}}) > 0$ for $N\geq 3$,
the reasoning of Section~\ref{counter} would suggest
that there is a positive chance that in fact
$\Phi^m(\gamma_1) \not\equiv \gamma_2 \pmod{p}$
for all $m$ and all primes $p$.

It is not difficult to construct explicit examples where this happens
if the orbit of $\gamma_1$ lies on a proper preperiodic subvariety of $\bP^N$
that does not contain $\gamma_2$; it would be interesting to find
examples where this happens when $\gamma_1$ has a Zariski dense
forward orbit.  Note also that Question~\ref{q1} has a negative
answer if $\Phi$ is not a morphism, or if it is a morphism of degree
one.

By a result of Fakhruddin \cite{Fa}, a positive answer to
Question~\ref{q1} for maps  $\Phi: \bP_K^N \lra \bP_K^N$ would give a
positive answer for any polarizable self-map $f: X \lra X$ of
projective varieties.
(A map $f:X \lra X$ is said to be \emph{polarizable} if there is
an ample divisor $L$ such that $f^*L \cong L^{\otimes d}$
for some $d > 1$; see \cite{ZhangLec}.)  In
particular, one would have a reasonable dynamical generalization of
\cite{Pink-abelian}.
However, proving that Question~\ref{q1} has a positive answer may require new
techniques.  It is not clear to us how to modify the arguments in this paper
to treat this higher-dimensional problem.

\textbf{Acknowledgements.}  R.B. gratefully
acknowledges the support of NSF grant DMS-0901494.
D.G. was partially supported by NSERC.
P.K. was supported in part by grants from the
G\"oran Gustafsson Foundation, the Knut and Alice Wallenberg
foundation, and the Swedish Research Council.
T.S. was partially supported by NSF grants DMS-0854998 and DMS-1001550.
T.T. was partially
supported by NSF grants DMS-0801072 and DMS-0854839.  Our
computations in Section~\ref{counter} were
done with CPU time provided by the Research Computing Cluster at the
CUNY Graduate Center.  The authors would like to thank Rafe Jones,
Adam Towsley, and Michael Zieve for helpful conversations,
and also thank the Centro di Ricerca Matematica Ennio De Giorgi for
its hospitality in the summer of 2009 when this project was started.

\bibliographystyle{amsalpha}
\bibliography{Drinbib}

\end{document}